\documentclass[a4paper, 11pt]{article}
\usepackage{amssymb}
\usepackage{amssymb}

\usepackage{amsmath,amssymb,amsthm}
\usepackage[latin1]{inputenc}
\usepackage{version,tabularx,multicol}
\usepackage{graphicx,float,psfrag}
\usepackage{stmaryrd}
 \usepackage{color}
\usepackage[pdftex]{hyperref}
\usepackage{amsfonts}
\usepackage{mathrsfs}
\usepackage{geometry}
\usepackage{longtable}
\usepackage[numbers,sort&compress]{natbib}

\theoremstyle{plain}
\newtheorem{theorem}{Theorem}[section]

 \newtheorem{lemma}[theorem]{Lemma}

\newtheorem{remark}{Remark}[section]

  \makeatletter
  \@addtoreset{equation}{section}
  \makeatother

 \def\beqlb{\begin{eqnarray}}\def\eeqlb{\end{eqnarray}}
 \def\beqnn{\begin{eqnarray*}}\def\eeqnn{\end{eqnarray*}}

 \def\qed{\hfill$\Box$\medskip}

\newcommand{\bcen}{\begin{center}}
\newcommand{\ecen}{\end{center}}
\newcommand{\bgeqn}{\begin{equation}}
\newcommand{\edeqn}{\end{equation}}

\begin{document}
\title{Upper deviation probabilities for the range of a supercritical super-Brownian motion}
\author{Shuxiong Zhang}
\maketitle

\renewcommand{\thefootnote}{\fnsymbol{footnote}}

\begin{abstract}

{Let $\{X_t\}_{t\geq 0 }$ be a $d$-dimensional supercritical super-Brownian motion started from the origin with branching mechanism $\psi$. Denote by $R_t:=\inf\{r>0:X_s(\{x\in \mathbb{R}^d:|x|\geq r\})=0,~\forall~0\leq s\leq t\}$ the radius of the minimal ball (centered at the
origin) containing the range of $\{X_s\}_{s\geq 0 }$ up to time $t$. In \cite{Pinsky}, Pinsky proved that condition on non-extinction, $\lim_{t\to\infty}R_t/t=\sqrt{2\beta}$ in probability, where $\beta:=-\psi'(0)$. Afterwards, Engl\"{a}nder \cite{Englander04} studied the lower deviation probabilities of $R_t$. For the upper deviation probabilities,  he \cite[Conjecture 8]{Englander04} conjectured that for $\rho>\sqrt {2\beta}$,

$$
\lim_{t\to\infty}\frac{1}{t}\log\mathbb{P}(R_t\geq \rho t)=-\left(\frac{\rho^2}{2}-\beta\right).
$$
In this note, we confirmed this conjecture.}
\end{abstract}



 \section{Introduction}
 In this paper, we consider a $d$-dimensional measure-valued Markov process $\{X_t\}_{t\geq0}$, called super-Brownian motion (henceforth SBM).  For convenience of the reader, we give a brief introduction to the SBM and some pertinent results needed in this article.
 \par
 Let $\ M_F(\mathbb{R}^d)$ be the finite measure space and $\psi$ be a function on $[0,\infty)$ of the form:
  \begin{align}\label{effergde}
\psi(u):=-\beta u+\alpha u^2+\int^{\infty}_0\left(e^{-uy}-1+uy\right)n(dy),~u\geq0,
\end{align}
where $\beta\in\mathbb{R},~\alpha\geq 0$ and $n$ is a $\sigma$-finite measure satisfying
$$
\int^{\infty}_0(y^2\wedge y) n(dy)<\infty.
$$
The function $\psi$ is called a branching mechanism.
  The SBM with initial value $\mu\in M_F(\mathbb{R}^d)$ and branching mechanism $\psi$ is an $M_F(\mathbb{R}^d)$-valued process, whose transition probabilities are characterized through their
Laplace transforms. Write $\nu(f):=\int_{\mathbb{R}^d}f(x)\nu(dx)$ for a (random) measure $\nu$ and positive function $f$. For all $\mu\in\ M_F(\mathbb{R}^d)$ and positive bounded function $g$ on $\mathbb{R}^d$,
\begin{align}\label{4t4rfr453}
\mathbb{E}_{\mu}\left[e^{-X_t(g)}\right]=e^{-\mu(v(t,x))},
\end{align}
where $v(t,x),~t\geq0,~x\in\mathbb{R}^d$ is the unique locally bounded positive solution of
\begin{align}\label{4t4tvf3}
\begin{cases}
\partial_t v=\frac{1}{2}\Delta v-\psi(v),\cr
v(0,x)=g(x).
\end{cases}
\end{align}
$\{X_t\}_{t\geq0}$ is called a supercritical (critical, subcritical) super-Brownian motion if $\beta>0$ $(=0,~<0).$ We refer the reader to \cite{Etheridge,Li11} for an explicit definition and more elaborate discussions on the super-Brownian motion.
In this paper, we only deal with the supercritical case, that is $\beta>0$. Moreover, to simplify the statement, we always assume $X_0=\delta_0$ in the rest of the article.

 \par
In this work, we are interested in the radius of the range of  $\{X_s\}_{s\geq0}$  up to time $t$, which is defined as
 $$R_t:=\inf\{r>0:X_s(B^c(r))=0,~\forall~0\leq s\leq t\},$$
where $B^c(r)$ stands for the complementary set of the ball $B(r):=\{x\in\mathbb{R}^d:|x|<r\}$. Denote by $S:=\{\forall t\geq0, X_t(\mathbb{R}^d)>0 \}$ the survival set of the SBM. It is well-known that under Grey's condition, $\mathbb{P}(S)=1-e^{-\lambda^*}$ (e.g. \cite[Theorem 3.8]{Li11}), where $\lambda^*\in(0,\infty)$ is the largest root of the equation $\psi(u)=0$.
\par
 In \cite[Theorems 1]{Englander04}, Engl\"{a}nder proved that if $\alpha>0,$~$n$ is a zero measure and $\rho\in(0,\sqrt {2\beta})$, then
\begin{align}\label{4t4emy}
\lim_{t\to\infty}\frac{1}{t}\log\mathbb{P}(R_t\leq \rho t|S)=-\beta+\sqrt{\beta/2}\rho,
\end{align}
which is the so-called lower deviation probabilities. He also \cite{Englander04} tried to consider relevant upper deviation probabilities. He found that the upper deviation probabilities of the branching Brownian motion can be deduced from Freidlin \cite{Freidlina}\cite{Freidlinb}. For the the upper deviation probabilities of the super-Brownian motion, he conjectured that it has the same asymptotics as the branching Brownian motion. But different ideas should be used. \par

In the following theorem, we confirm this conjecture.

\begin{theorem}\label{th5} Suppose that there exists some constant $\gamma>0$ such that
\begin{align}\label{r4fk93}
\int^{\infty}_1 y(\log y)^{2+\gamma}n(dy)<\infty.
\end{align}
Assume $\alpha>0$ and $\rho\in(\sqrt {2\beta},\infty)$. Then,
\begin{align}\label{4rer43r4f}
\lim_{t\to\infty}\frac{1}{t}\log\mathbb{P}(R_t\geq \rho t)=-\left(\frac{\rho^2}{2}-\beta\right).
\end{align}
\end{theorem}

\begin{remark}

 Condition (\ref{r4fk93}) comes from Ren et al. \cite[Theorem 1.2]{Ren21}. The reason is that, in the proof of Theorem \ref{th5} (see (\ref{4trte34f1}) below), we in virtue of the maximum of the one-dimensional super-Brownian motion to derive the lower bound. The assumption $\alpha>0$ is because in the proof of upper bound, we use the lower deviation probabilities of $R_t$. Nevertheless, the assumptions on the branching mechanism is weaker than Pinsky \cite{Pinsky} and Engl\"{a}nder \cite{Englander04}, who assumed $\psi(u)=-\beta u+\alpha u^2$ with $\beta,~\alpha>0$.
\end{remark}

\begin{remark}
 In fact, besides using \cite[Theorem 1.2]{Ren21}, one can also obtain the precise lower bound of Theorem \ref{th5} easily through the comparison lemma; see \cite[Proposition 4]{Englander04}. So, the difficulty lies in the upper bound. To overcome it, we use the the weighted occupation time of the super-Brownian motion to characterize $R_t$, and then the precise upper bound is obtained through the Feynman-Kac formula and lower deviation probabilities of $R_t$.
\end{remark}

\section{Proof of Theorem \ref{th5}}
In this section, we are going to prove Theorem \ref{th5}. Let's give a brief discussion about the
strategy of proofs.  The proof of the theorem is divided into two parts--lower bound and upper bound. The lower bound is easy by using the maximum to control the radius of the range. For the upper bound, we use partial differential equations to portray probabilities, and establish a Feynman-Kac formula to estimate the upper bound of the solution of the PDE. To be specific, by using a stopping time strategy to the Feynman-Kac formula, we found the solution of the PDE can be bounded above be the by the deviation probabilities of the Brownian motion and lower deviation probabilities of $R_t$.
\par
To establish the Feynman-Kac formula, we should first give a formulation of the weighted occupation time $\int^t_0X_s(\phi)ds$, which was first introduced by Iscoe \cite{Iscoe86} for the super-stable process. He then used it to study the supporting properties of super-Brownian motions \cite{Iscoe88}. Li \cite[Section 5.4]{Li11} gives a detailed introduction of the weighted occupation time of a general superprocess.
\par
Let $\mathbb{P}_{\delta_x}$ ($\mathbb{P}_{x}$) be the probability measure under which the SBM $\{X_t\}_{t\geq0}$ (the $d$-dimensional standard Brownian motion $\{B_t\}_{t\geq0}$) starts from the Dirac measure (the point) at $x$. Moreover, we write $\mathbb{P}_{\delta_x}=\mathbb{P}$ and $\mathbb{P}_{x}=\mathbb{P}$ for short. Assume $\phi(x)$ and $g(x)$ are non-negative, bounded and continuous function on $\mathbb{R}^d$. From \cite[Corollary 5.17, Theorem 7.12]{Li11}, we have
\begin{align}\label{4t4rfr453}
\mathbb{E}_{\delta_x}\left[e^{-X_t(g)-\int^t_0X_s(\phi)ds}\right]=e^{-v(t,x)},
\end{align}
where $v(t,x),~t\geq0,~x\in\mathbb{R}^d$ is the unique locally bounded positive solution of
\begin{align}\label{4t4tvf3}
\begin{cases}
\partial_t v=\frac{1}{2}\Delta v-\psi(v)+\phi,\cr
v(0,x)=g(x).
\end{cases}
\end{align}
\par
Now, we are ready to present the Feynman-Kac formula to above partial differential equation. Set
$$k(v):=\frac{\psi(v)}{v},$$
where we make the convention that $k(0)=-\beta$ if $v=0$.
\begin{lemma}\label{6mk1q}(Feynman-Kac formula) Let $v(t,x), t\geq 0, x\in\mathbb{R}^d$ be the unique solution of (\ref{4t4tvf3}). Then, we have
\begin{align}
v(t,x)=\mathbb{E}_x\left[g(B_t)e^{-\int^t_0 k(v(t-r,B_r))dr}+\int^t_0\phi(B_s)e^{-\int^{s}_0k(v(t-r,B_r))dr}ds\right].
\end{align}
\end{lemma}
\begin{proof}
We first cast (\ref{4t4tvf3}) into the following mild form (see \cite[Corollary 5.17]{Li11}):
\begin{align}
v(t,x)=\mathbb{E}_{x}\left[g(B_t)\right]+\int^t_0 \mathbb{E}_{x}[\phi(B_s)]ds-\mathbb{E}_{x}\left[ \int^t_0 \psi(v(t-s,B_s))ds\right],
\end{align}
which implies
\begin{align}\label{454t4f}
v(t,x)\leq \| g\|_{\infty} +\| \phi\|_{\infty}t-\left(\min_{u\geq 0}\psi(u)\right) t.
\end{align}
Let $\nabla_x$ be the gradient operator. By It$\hat{\text{o}}$'s formula,
\begin{align}
&d\left[v(t-s, B_s)e^{-\int^s_0 k(v(t-r,B_r))dr}\right]\cr
&=e^{-\int^s_0 k(v(t-r,B_r))dr}dv(t-s, B_s)+v(t-s, B_s)de^{-\int^s_0 k(v(t-r,B_r))dr}\cr
&=e^{-\int^s_0 k(v(t-r,B_r))dr}\left[-\partial_tv(t-s,B_s)ds+\nabla_xv(t-s,B_s)dB_s+\frac{1}{2}\Delta v(t-s,B_s)ds\right]\cr
&~~~-e^{-\int^s_0 k(v(t-r,B_r))dr}\psi(v(t-s, B_s))ds\cr
&=e^{-\int^s_0 k(v(t-r,B_r))dr}\nabla_xv(t-s,B_s)dB_s-e^{-\int^s_0 k(v(t-r,B_r))dr}\phi(B_s).\nonumber
\end{align}
Integrating on both sides gives that
\begin{align}
&v(t-s, B_s)e^{-\int^s_0 k(v(t-r,B_r))dr}-v(t,x)\cr
&=\int^s_0e^{-\int^{s'}_0 k(v(t-r,B_r))dr}\partial_xv(t-s',B_s')dB_{s'}-\int^s_0e^{-\int^{s'}_0 k(v(t-r,B_r))dr}\phi(B_{s'})ds'.
\end{align}
For $s\in[0,t]$, set
\begin{align}
M_s&:=\int^s_0e^{-\int^{s'}_0 k(v(t-r,B_r))dr}\partial_xv(t-s',B_s')dB_{s'}\cr
&=v(t-s, B_s)e^{-\int^s_0 k(v(t-r,B_r))dr}-v(t,x)+\int^s_0e^{-\int^{s'}_0 k(v(t-r,B_r))dr}\phi(B_{s'})ds'.
\end{align}
This, combined with (\ref{454t4f}), implies that $M_s,~s\in[0,t]$ is a bounded martingale. Therefore,
\begin{align}
&\mathbb{E}_{x}\left[g(B_t)e^{-\int^t_0 k(v(t-r,B_r))dr}-v(t,x)+\int^t_0e^{-\int^{s}_0 k(v(t-r,B_r))dr}\phi(B_{s})ds\right]\cr
&=\mathbb{E}_{x}[M_t]=\mathbb{E}_{x}[M_0]=0,\nonumber
\end{align}
which concludes the lemma.
\end{proof}
\par
At this moment, we are ready to prove Theorem \ref{th5}.
\par
\textbf{Proof of Theorem \ref{th5}.} \textbf{Lower bound.} Let $H_t$ be the right-most position of the one-dimensional super-Brownian motion at time $t$. It is simple to see that
\begin{align}\label{4trte34f1}
\mathbb{P}(R_t\geq \rho t)\geq \mathbb{P}(H_t\geq \rho t).
\end{align}
Then, the desired lower bound follows by using \cite[(1.23)]{Ren21}. In fact, \cite[(1.23)]{Ren21} also needs the assumption that there exists some $\theta\in(0,1]$ and $a,~b>0$ such that
$$
\psi(u)\geq -au+bu^{1+\theta}~\text{for}~u\geq0.
$$
But this condition is satisfied, since we have assumed $\alpha>0$.
\par
\textbf{Upper bound.} Let $N$ be a $d$-dimensional standard normal random vector. Then, by the polar coordinates transformation, we have for $z>0$,
\begin{align}
\mathbb{P}(|N|\geq z)&=\int_{|y|\geq z}\frac{1}{(2\pi)^{d/2}}e^{-\frac{|y|^2}{2}}dy\cr
&=\int^{\infty}_{z}\frac{dw_d}{(2\pi)^{d/2}}e^{-\frac{r^2}{2}}r^{d-1}dr,
\end{align}
where $w_d$ is the volume of a $d$-dimensional unit ball. Thus, by L'H$\hat{\text{o}}$pital's rule, we have
$$
\lim_{z\to\infty}\frac{\mathbb{P}(|N|\geq z)}{e^{-\frac{z^2}{2}}z^{d-2}}=1.
$$
Thus, there exists a positive constant $C_d$ depending only on $d$ such that for any $z>1$,
\begin{align}\label{ii3refed}
\mathbb{P}(|N|\geq z)\leq C_d e^{-\frac{z^2}{2}}z^{d-2}.
\end{align}
 For $M>0$, let $f(x)$ be a smooth and non-negative function on $\mathbb{R}^d$ such that $\{x\in\mathbb{R}^d:f(x)>0\}=B^c(M)$. Put $\phi_{\lambda}(x):=\lambda f(x)$ for $\lambda>0$. Denote by $v_{\lambda}(t,x)$ the solution of (\ref{4t4tvf3}) with $\phi=\phi_{\lambda}$, $g=0$. Since $X_s(\phi_{\lambda})$ is right continuous with respect to $s$, we have
\begin{align}\label{4tmjk1}
\mathbb{P}(R_t\geq M)&=\mathbb{P}(\exists s\in [0,t], X_s(B^c(M))>0)\cr
&=\mathbb{P}(\exists s\in [0,t], X_s(\phi_{\lambda})>0)\cr
&=\mathbb{P}\left(\int^{t}_{0}X_s(\phi_{\lambda})ds>0\right)\cr
&=1-\lim_{\lambda\to\infty}\mathbb{E}\left[e^{-\int^{t}_{0}X_s(\phi_{\lambda})ds}\right]\cr
&=1-\lim_{\lambda\to\infty}e^{-v_{\lambda}(t,0)}\cr
&\leq\lim_{\lambda\to\infty}v_{\lambda}(t,0).
\end{align}
Let $u_{\lambda}(t,x)$ be the unique solution of
\begin{align}
\begin{cases}
\partial_t u=\frac{1}{2}\Delta u-(-\beta u+\alpha u^2)+\phi_{\lambda},\cr
u(0,x)=0.
\end{cases}
\end{align}
Since
\begin{align}
\partial_t u-\frac{1}{2}\Delta u+\psi(u)&=\psi(u)-[-\beta u+\alpha u^2]+\phi_{\lambda}\cr
&\geq \phi_{\lambda}\cr
&=\partial_t v-\frac{1}{2}\Delta v+\psi(v),&
\end{align}
by the maximum priciple (see \cite[Lemma 2.2]{Ren21}), we have for any $t,~\lambda\geq0$ and $x\in\mathbb{R}^d$,
$$v_{\lambda}(t,x)\leq u_{\lambda}(t,x).$$
This means, in the following proofs, we can assume $n$ is a zero measure in (\ref{effergde}). Namely, without loss of generality, we assume
$$k(v)=-\beta +\alpha v.$$
 By Lemma \ref{6mk1q},
\begin{align}\label{4r43eb}
v_{\lambda}(t,x)=\int^t_{0} e^{\beta s}\mathbb{E}_x\left[\phi_{\lambda}(B_{s})e^{-\int^{s}_0 \alpha v_{\lambda}(t-r,B_r)dr}\right]ds.
\end{align}
In the following proofs, we sometimes write $v(t,x):=v_{\lambda}(t,x)$ for short. For $\varepsilon\in(0,1\wedge(\sqrt{\beta}/\rho))$, let $T$ be the first time of $\{B_s\}_{s\geq 0}$ to hit the boundary of  $B((1-\varepsilon)M)$, i.e.,
$$T:=\inf\{s\geq0:|B_s|=(1-\varepsilon)M\}.$$
Let $\{B'_s\}_{s\geq0}$ be a independent Brownian motion  of $\{B_s\}_{s\geq 0}$. Note that if $s\leq T$, then $\phi_{\lambda}(B_{s})=0$. This, together with (\ref{4r43eb}), yields that
\begin{align}\label{4r4erm}
v(t,0)&=\int^t_{0} e^{\beta s}\mathbb{E}_0\left[\phi_{\lambda}(B_{s})e^{-\int^{s}_0 \alpha v(t-r,B_r)dr}\right]ds\cr
&=\int^t_{0} e^{\beta s}\mathbb{E}_0\left[\mathbf{1}_{\{T<s\}}\phi_{\lambda}(B_{s})e^{-\int^{s}_0 \alpha v(t-r,B_r)dr}\right]ds\cr
&=\int^t_{0} e^{\beta s}\mathbb{E}_0\left[\mathbf{1}_{\{T<s\}}e^{-\int^{T}_0 \alpha v(t-r,B_{r})dr}\mathbb{E}_{B_T}\left[\phi_{\lambda}(B'_{s-T})e^{-\int^{s}_T \alpha v(t-r,B'_{r-T})dr}\right]\right]ds\cr
&\leq\mathbb{E}_0\left[\int^t_{0} e^{\beta s}\mathbf{1}_{\{T<s\}}\mathbb{E}_{B_T}\left[\phi_{\lambda}(B'_{s-T})e^{-\int^{s}_T \alpha v(t-r,B'_{r-T})dr}\right]ds\right],
\end{align}
where third equality follows from the strong Markov property of Brownian motion. By the Fubini's theorem, (\ref{4r4erm}) yields that
\begin{align}\label{45lpia}
v(t,0)
&\leq\mathbb{E}_0\left[\mathbf{1}_{\{T\leq t\}}\int^t_{T} e^{\beta s}\mathbb{E}_{B_T}\left[\phi_{\lambda}(B'_{s-T})e^{-\int^{s-T}_0 \alpha v(t-T-r,B'_{r})dr}\right]ds\right]\cr
&=\mathbb{E}_0\left[\mathbf{1}_{\{T\leq t\}}\int^{t-T}_{0} e^{\beta (T+\eta)}\mathbb{E}_{B_T}\left[\phi_{\lambda}(B'_{\eta})e^{-\int^{\eta}_0\alpha v(t-T-r,B'_{r})dr}\right]d\eta\right]\cr
&=\mathbb{E}_0\left[\mathbf{1}_{\{T\leq t\}}e^{\beta T}v_{\lambda}(t-T,B_T)\right].
\end{align}
Since $g=0$, from (\ref{4t4rfr453}), it is simple to see that $v_{\lambda}(t,x)$ is increasing with respect to $\lambda$ and $t$. Thus, if we let $M=\rho t$, then there exists a positive constant $C_{\varepsilon,\rho,\beta}$ depending only on $\varepsilon,~\rho$ and $\beta$ such that for $t$ large enough,
\begin{align} \label{3r3re3}
v_{\lambda}(t-T,B_T)&\leq v_{\lambda}(t,B_T)\cr
&\leq \lim_{\lambda\to\infty}v_{\lambda}(t,B_T)\cr
&=-\log\mathbb{P}_{B_T}(\forall s\in[0,t], X_s(B^c(M))=0)\cr
&\leq-\log\mathbb{P}_{\delta_0}(R_t\leq \varepsilon \rho t)\cr
&\leq -\log\left[\mathbb{P}_{\delta_0}(R_t\leq \varepsilon \rho t|S)\mathbb{P}(S)\right]\cr
&\leq C_{\varepsilon,\rho,\beta}t,
\end{align}
where the first equality follows from similar arguments to obtain the last equality of (\ref{4tmjk1}) and the last inequality comes from (\ref{4t4emy}). For $1\leq i\leq d$, let $B_s^{(i)}$ be the $i$-th component of the vector $B_s$. It is easy to see that
\begin{align}
\max_{s\in [0,t]}|B_s|&=\max_{s\in [0,t]}\left[\sum^d_{i=1}(B^{(i)}_s)^2\right]^{1/2}\cr
&\leq \left[\sum^d_{i=1}(\max_{s\in [0,t]} |B^{(i)}_s|)^2\right]^{1/2}\cr
&= \left[\sum^d_{i=1}\left[(\max_{s\in [0,t]} B^{(i)}_s)\vee (-\min_{s\in [0,t]} B^{(i)}_s)\right ]^2\right]^{1/2}.
\end{align}
Since $\max_{s\in [0,t]} B^{(i)}_s$ and $-\min_{s\in [0,t]} B^{(i)}_s$ have the same distribution $|B^{i}_t|$, it follows as above that if $(1-\varepsilon) \rho\sqrt {t}>1$, then
\begin{align}\label{3rryux}
&\mathbb{P}(T\leq t)\cr
 &=\mathbb{P}\left(\max_{s\in [0,t]}|B_s|\geq (1-\varepsilon)M\right)\cr
 &\leq 2^{d}\mathbb{P}\left(\left[\sum^d_{i=1}\left(\max_{s\in [0,t]} B^{(i)}_s\right)^2\right]^{1/2} \geq (1-\varepsilon)M \right)\cr
 &= 2^{d}\mathbb{P}\left(\left[\sum^d_{i=1}\left(B^{(i)}_t\right)^2\right]^{1/2} \geq (1-\varepsilon)M \right)\cr
 &= 2^{d}\mathbb{P}\left(|B_t| \geq (1-\varepsilon)M \right)\cr
 &= 2^{d}\mathbb{P}\left(|B_t|/\sqrt t \geq (1-\varepsilon)\rho\sqrt {t} \right)\cr
 &\leq  2^{d} C_d e^{-\frac{[(1-\varepsilon)\rho\sqrt {t}]^2}{2}}[(1-\varepsilon) \rho\sqrt {t}]^{d-2},
\end{align}
where the last inequality comes from (\ref{ii3refed}). Plugging (\ref{3r3re3}) and (\ref{3rryux}) into (\ref{45lpia}) yields that for $t$ large enough,
\begin{align}\label{454r6n}
v(t,0)&\leq e^{\beta t} \mathbb{P}\left(T\leq t\right)C_{\varepsilon,\rho,\beta}t\cr
&\leq e^{\beta t}2^{d} C_d e^{-\frac{\left[(1-\varepsilon)\rho\sqrt {t}\right]^2}{2}}\left[(1-\varepsilon)\rho\sqrt {t}\right]^{d-2}C_{\varepsilon,\rho,\beta}t.
\end{align}
Plugging (\ref{454r6n}) into (\ref{4tmjk1}) and then taking limits yields that
$$
\limsup_{t\to\infty}\frac{1}{t}\log\mathbb{P}(R_t\geq \rho t)\leq-\left(\frac{(1-\varepsilon )^2{\rho}^2}{2}-\beta \right).
$$
The desired upper bound follows by letting $\varepsilon\to0$.\qed




\begin{thebibliography}{99}

%
%
%
%
%

%
%
%
%










\bibitem{Englander04}
J. Engl\"{a}nder.
\newblock{Large deviations for the growth rate of the support
of supercritical super-Brownian motion}
\newblock{\em Statistics and Probability Letters}, 66(4):449-456, 2004.

\bibitem{Etheridge}
A. M. Etheridge.
\newblock{\em An introduction to superprocess.}
\newblock{University Lecture Series, 20. American
Mathematical Society, Providence, RI, 2000.}

\bibitem{Freidlina}
M. Freidlin.
Limit theorems for large deviations and reaction-diffusion equations.
 \newblock{\em The Annals of Probability}, 13:639-675, 1985.

 \bibitem{Freidlinb}
M. Freidlin.
\newblock{\em Functional integration and partial differential equations.}
 In: Annals of Mathematics Studies, Vol. 109.
Princeton University Press, Princeton, NJ.

\bibitem{Iscoe86}
I. Iscoe.
 \newblock A weighted occupation time for a class of measure-valued branching processes. {\em Probability Theory and Related Fields}, 71:85-116, 1986.

\bibitem{Iscoe88}
I. Iscoe.
 On the supports of measure-valued critical branching Brownian motion.
  {\em The Annals of Probability}, 16:200-221, 1988.



\bibitem{Li11}
Z. Li.
\newblock {\em Measure-valued Branching Markov Processes.}
 Probability and its Applications (New York).
\newblock{Springer, Heidelberg,} 2011.



\bibitem{Pinsky}
R. G. Pinsky.
\newblock{On the large time growth rate of the support of supercritical super-Brownian motion}
\newblock{\em The Annals of Probability}, 23:1748-1754, 1995.

\bibitem{Ren21}
Y. Ren, R. Song and R. Zhang.
\newblock{ The extremal process of super-Brownian motion.}
\newblock{\em  Stochastic Processes and their Applications}, 137:1-34, 2021.




\end{thebibliography}


\vspace{2cm}
\noindent{\bf ~~E-mail:}~shuxiong.zhang@mail.bnu.edu.cn\\
{\bf Address:}~School of Mathematics and Statistics, Anhui Normal University, Wuhu, China.

\end{document}